\documentclass[
submission
]{dmtcs-episciences}

\usepackage[utf8]{inputenc}
\usepackage{subfigure}

\usepackage{amsthm,amssymb,amsmath,cite,float,enumerate}
\numberwithin{equation}{section}
\newtheorem{theorem}{Theorem}
\newtheorem{lemma}[theorem]{Lemma}

\newtheorem{corollary}[theorem]{Corollary}

\newtheorem{claim}{Claim}

\author{Julien Baste\affiliationmark{1}
  \and Stefan Ehard\affiliationmark{2}
  \and Dieter Rautenbach\affiliationmark{2}}
\title[Non-monotone target sets for threshold values restricted to $0$, $1$, and the vertex degree]{Non-monotone target sets for threshold values restricted to $0$, $1$, and the vertex degree\thanks{Funded by the Deutsche Forschungsgemeinschaft (DFG, German Research Foundation) - 388217545.}}
\affiliation{
Univ. Lille, CNRS, Centrale Lille, UMR 9189 CRIStAL, Lille, France\\
Institute of Optimization and Operations Research, Ulm University, Germany}
\received{2020-10-19}
\revised{2022-02-18}
\accepted{2022-04-28}
\begin{document}
\publicationdetails{24}{2022}{1}{21}{6844}
\maketitle
\begin{abstract}
We consider a non-monotone activation process $(X_t)_{t\in\{ 0,1,2,\ldots\}}$
on a graph $G$, where 
$X_0\subseteq V(G)$, 
$X_t=\{ u\in V(G):|N_G(u)\cap X_{t-1}|\geq \tau(u)\}$ 
for every positive integer $t$, and 
$\tau:V(G)\to \mathbb{Z}$ is a threshold function.
The set $X_0$ is a so-called non-monotone target set for $(G,\tau)$
if there is some $t_0$ such that $X_t=V(G)$ for every $t\geq t_0$.
Ben-Zwi, Hermelin, Lokshtanov, and Newman [Discrete Optimization 8 (2011) 87-96] 
asked whether a target set of minimum order 
can be determined efficiently if $G$ is a tree.
We answer their question in the affirmative
for threshold functions $\tau$ satisfying 
$\tau(u)\in \{ 0,1,d_G(u)\}$ for every vertex~$u$.
For such restricted threshold functions,
we give a characterization of target sets that allows to show
that the minimum target set problem 
remains NP-hard for planar graphs of maximum degree $3$
but is efficiently solvable for graphs of bounded treewidth.
\end{abstract}

\section{Introduction}

Target sets are a widely studied model for spreading processes in networks, such as influence diffusion and spread of opinions in social networks or the spread of an infectious disease.
For a graph $G$ and an integer-valued threshold function $\tau$ on its vertices, a \emph{target set} is a set of vertices of $G$ that we consider \emph{active}, and by iteratively activating vertices $v$ of $G$ that have at least $\tau(v)$ active neighbours, eventually the entire vertex set of $G$ becomes active.
This monotone version --- as activated vertices remain active for the entire process --- has received most attention \cite{acbewo, beehpera, cedoperasz,ch,ehra,gera,za} and has been studied in various variations \cite{keklta,drro}.

In this paper we study the natural non-monotone target set selection problem as described by Ben-Zwi, Hermelin, Lokshtanov, and Newman~\cite{behelonw}, where a vertex $v$ of $G$ becomes non-active at any iteration of the spreading process whenever the number of its active neighbours is less than $\tau(v)$. 
Vertices may activate and deactivate several times,
and thus, the underlying process is non-monotone.
Unsurprisingly, the optimization problem of finding a minimum non-monotone target set is notably hard; Ben-Zwi et al.~\cite{behelonw} show $\#$P-hardness of a weighted directed version.
Surprisingly, it even remains open whether the (unweighted and undirected)
non-monotone target set selection problem
can be solved efficiently  on trees --- a question that was raised in 2011 by Ben-Zwi et al.~\cite{behelonw}.
In this paper we make the first moderate progress on this question.
Our results grew out of an efficient solution for paths
and apply to a natural class of restricted instances.
Before we collect some terminology and notation in order to state our results,
we would like to point out that non-monotone processes
were also studied in \cite{dodrrasz,pe}.

We consider finite, simple, and undirected graphs.
The sets of positive integers and of non-negative integers 
are denoted by 
$\mathbb{N}=\{ 1,2,3,\ldots\}$
and
$\mathbb{N}_0=\{ 0,1,2,3,\ldots\}$, respectively.
For an integer $n$, let $[n]$ be the set of positive integers at most $n$.
Let $G$ be a graph.
For a set $X$ of vertices of $G$, let 
$N_G(X)=
\left(\bigcup_{u\in X}N_G(u)\right)\setminus X$, 
and let
$N_G[X]=X\cup N_G(X)$.
A {\it threshold function for} $G$ is a function $\tau:V(G)\to\mathbb{N}_0$.
If $X$ is a set of vertices of $G$, then the {\it non-monotone activation process on $(G,\tau)$ starting with $X$}
is the sequence $(X_t)_{t\in\mathbb{N}_0}$, 
$X=X_0$, and
$$X_t=\big\{ u\in V(G):|N_G(u)\cap X_{t-1}|\geq \tau(u)\big\}\mbox{ for every $t$ in $\mathbb{N}$.}$$
If there is some $t_0\in\mathbb{N}_0$ 
such that $X_t=V(G)$ for every $t\geq t_0$,
then $X$ is a {\it non-monotone target set for $(G,\tau)$};
note that $t_0\leq 2^n$ if $t_0$ exists and $G$ has order $n$.
If $\tau(u)>d_G(u)$ for some vertex $u$ of $G$,
where $d_G(u)$ is the degree of $u$ in $G$,
then $u\not\in X_t$ for every $t$ in $\mathbb{N}$.
Therefore, we may assume $\tau\leq d_G$ in what follows.
Note, furthermore, that vertices $u$ with $\tau(u)<0$
behave similarly as vertices with $v$ with $\tau(v)=0$.
Hence, we may additionally assume $\tau\geq 0$ in what follows.

Our results concern the non-monotone target set problem for instances $(G,\tau)$,
where 
\begin{eqnarray}\label{erest}
\tau(u)\in \big\{ 0,1,d_G(u)\big\}\mbox{ for every vertex $u$ of $G$.}
\end{eqnarray}
First, we describe a simple reduction for such instances in Lemma~\ref{lemma1},
which isolates the vertices~$u$ with $\tau(u)=d_G(u)$.
Our central result is Theorem~\ref{theorem1}, 
which characterizes non-monotone target sets for such instances
in terms of intersection conditions. 
With Theorem~\ref{theorem2} 
we show that the considered restricted instances are still hard. 
Finally,
with Theorem~\ref{theorem3}
and Corollary~\ref{corollary1},
we show that the considered restricted instances 
are tractable for graphs of bounded treewidth;
providing a positive answer to the stated question from \cite{behelonw}
at least for the considered restricted instances.

\section{Results}

Throughout the paper, component always means connected component.
Our first lemma yields a simple reduction rule.

\begin{lemma}\label{lemma1}
Let $G$ be a graph.
Let $\tau$ be a threshold function for $G$ 
satisfying (\ref{erest}).
Let $X$ be a set of vertices of $G$.
Let $U$ be the vertex set of a component 
of order at least $2$ of the graph 
$$G\big[\big\{ u\in V(G):\tau(u)=d_G(u)\big\}\big].$$
$X$ is a non-monotone target set for $(G,\tau)$
if and only if 
$N_G[U]\subseteq X$
and
$X\setminus N_G[U]$ is a non-monotone target set for $(G',\tau')$,
where 
\begin{eqnarray*}
G'&=&G-N_G[U]\mbox{ and }\\
\tau'(u)&=&\max\big\{ 0,\tau(u)-|N_G(u)\cap N_G[U]|\big\}\mbox{ for every vertex $u$ of $G'$.}
\end{eqnarray*}
Furthermore, 
$(G',\tau')$ satisfies (\ref{erest}),
that is, $\tau'(u)\in \big\{ 0,1,d_{G'}(u)\big\}$ for every $u\in V(G')$.
\end{lemma}
\begin{proof}
We first prove the necessity part of the stated equivalence.
Therefore, let $X$ be a non-monotone target set for $(G,\tau)$.
Let $(X_t)_{t\in\mathbb{N}_0}$ be the non-monotone activation process on $(G,\tau)$ starting with $X$.
If $u\not\in X_t$ for some $u\in U$ and some $t\in\mathbb{N}_0$, and $v$ is a neighbor of $u$ in~$U$,
then $v\not\in X_{t+1}$ and $u\not\in X_{t+2}$, 
which implies the contradiction that $u\not\in X_{t+2k}$ for every $k\in\mathbb{N}_0$.
It follows that $U\subseteq X_t$ for every $t\in\mathbb{N}_0$,
and, in particular, $U\subseteq X$.
If $u\not\in X$ for some $u\in N_G(U)$, and $v$ is a neighbor of $u$ in $U$,
then $v\not\in X_1$, which is a contradiction.
Hence, we obtain $N_G(U)\subseteq X$.
Altogether, it follows that $N_G[U]\subseteq X$, which, in view of the $\tau$ values,
implies that $N_G[U]\subseteq X_t$ for every $t\in\mathbb{N}_0$.
Furthermore, if $(X'_t)_{t\in\mathbb{N}_0}$ is the non-monotone activation process on $(G',\tau')$ starting with $X\setminus N_G[U]$,
then the definitions of $G'$ and $\tau'$ imply that $X_t=X_t'\cup N_G[U]$ for every $t\in \mathbb{N}_0$,
which implies that $X\setminus N_G[U]$ is a non-monotone target set for $(G',\tau')$.

Now, we prove the sufficiency part of the stated equivalence.
Therefore, let $N_G[U]\subseteq X$
and let $X\setminus N_G[U]$ be a non-monotone target set for $(G',\tau')$.
Let $(X_t)_{t\in\mathbb{N}_0}$ and $(X'_t)_{t\in\mathbb{N}_0}$ be as above.
In view of the $\tau$ values,
it follows that $N_G[U]\subseteq X_t$ for every $t\in\mathbb{N}_0$.
By the definitions of $G'$ and $\tau'$,
this implies $X_t=X_t'\cup N_G[U]$ for every $t\in \mathbb{N}_0$,
which implies that $X$ is a non-monotone target set for $(G,\tau)$.

Let $u$ be a vertex of $G'$.
If $\tau(u)\in \{ 0,1\}$, then $\tau'(u)\in \{ 0,1\}$ follows immediately.
If $\tau(u)=d_G(u)$ and $\tau'(u)\not=0$, then 
$\tau'(u)
=\tau(u)-|N_G(u)\cap N_G[U]|
=d_G(u)-|N_G(u)\cap N_G[U]|
=d_{G'}(u)$.
Altogether, the function $\tau'$ satisfies (\ref{erest}).
\end{proof}

For our next result, we may assume that the polynomial time reduction 
described in Lemma~\ref{lemma1}
has already been applied.

\begin{theorem}\label{theorem1}
Let $G$ be a graph.
Let $\tau$ be a threshold function for $G$ satisfying (\ref{erest})
such that $G$ has no edge $uv$ with $\tau(u)=d_G(u)$ and $\tau(v)=d_G(v)$.
A set $X$ of vertices of $G$ is a non-monotone target set for $(G,\tau)$
if and only if the following conditions hold:
\begin{enumerate}[(1)]
\item Let $U$ be the vertex set of a component of $G\big[\big\{ u\in V(G):\tau(u)\in \{ 0,1\}\big\}\big]$ 
with $|U|\geq 2$ and $\tau(u)=1$ for every $u\in U$.
\begin{enumerate}[(a)]
\item If $G[U]$ is not bipartite, then $X\cap N_G[U]\not=\emptyset$.
\item If $G[U]$ is bipartite with partite sets $A$ and $B$,
$A'=N_G(B)\setminus N_G[A]$,
$B'=N_G(A)\setminus N_G[B]$, and
$C=N_G(A)\cap N_G(B)$, then
$$X\cap (A\cup A'\cup C)\not=\emptyset\,\,\,\,\,\,\,\mbox{ and }\,\,\,\,\,\,\,X\cap (B\cup B'\cup C)\not=\emptyset.$$
\end{enumerate}
\item If $u$ is a vertex with $\tau(u)=1$ such that $\tau(v)=d_G(v)$ for every $v\in N_G(u)$,
then there are two (not necessarily distinct) neighbors $v_1$ and $v_2$ of $u$ such that 
$v_1\in X$ and $N_G(v_2)\subseteq X$.
\end{enumerate}
\end{theorem}
\begin{proof}
We first prove the necessity.
Let $X$ be a non-monotone target set for $(G,\tau)$.
Let $(X_t)_{t\in\mathbb{N}_0}$ be the non-monotone activation process on $(G,\tau)$ starting with $X$.
Let $U$ be as in (1).
Note that $\tau(u)=d_G(u)$ for every vertex $u\in N_G(U)$.
Therefore, if $X\cap N_G[U]=\emptyset$, then $X_t\cap N_G[U]=\emptyset$ for every $t\in \mathbb{N}_0$,
which is a contradiction.
This already implies condition (1)(a).
Now, let $U$, $A$, $B$, $A'$, $B'$, and $C$ be as in (1)(b). 
If $X_t\cap (A\cup A'\cup C)=\emptyset$ for some $t\in\mathbb{N}_0$, then, since 
\begin{itemize}
\item $N_G(B)\subseteq A\cup A'\cup C$, and
\item every vertex in $B'\cup C$ has a neighbor in $A$,
\end{itemize}
we have 
$X_{t+1}\cap (B\cup B'\cup C)=\emptyset$, and, by symmetry,
$X_{t+2}\cap (A\cup A'\cup C)=\emptyset$,
which implies the contradiction that $X_{t+2k}\cap (A\cup A'\cup C)=\emptyset$ for every $k\in\mathbb{N}_0$.
By symmetry, it follows that condition (1)(b) holds.
Now, let $u$ be as in (2).
If $u\not\in X_t$ for some $t\in\mathbb{N}_0$, then $X_{t+1}\cap N_G(u)=\emptyset$ and $u\not\in X_{t+2}$,
which implies the contradiction that $u\not\in X_{t+2k}$ for every $k\in\mathbb{N}_0$.
Hence, $u\in X_t$ for every $t\in\mathbb{N}_0$.
Since $u\in X_1$, there is a neighbor $v_1$ of $u$ with $v_1\in X$.
Since $u\in X_2$, there is a neighbor $v_2$ of $u$ with $v_2\in X_1$, which implies that $N_G(v_2)\subseteq X$.
Altogether, condition (2) follows.

Now, we prove the sufficiency.
Therefore, let $X$ satisfy conditions (1) and (2).
Let $(X_t)_{t\in\mathbb{N}_0}$ be the non-monotone activation process on $(G,\tau)$ starting with $X$.
If there is some $t_0$ such that $\big\{ u\in V(G):\tau(u)\in \{ 0,1\}\big\}\subseteq X_t$ for every $t\geq t_0$,
then $X_t=V(G)$ for every $t\geq t_0+1$. Therefore, it suffices to show the existence of $t_0$.
If $\tau(u)=0$ for some vertex $u$, then $u\in X_t$ for every $t\in\mathbb{N}$.
Now, let $u$, $v_1$, and $v_2$ be as in (2).
Note that $\tau(w)\in\{ 0,1\}$ for every $w\in N_G(v_1)\cup N_G(v_2)$.
If $\{ v_1\}\cup N_G(v_2)\subseteq X_t$ for some $t\in\mathbb{N}_0$,
we obtain $\{ v_2\}\cup N_G(v_1)\subseteq X_{t+1}$, and 
$\{ v_1\}\cup N_G(v_2)\subseteq X_{t+2}$, which implies $u\in X_t$ for every $t\in\mathbb{N}_0$.

Now, let $U$ be the vertex set of a component of $G\big[\big\{ u\in V(G):\tau(u)\in \{ 0,1\}\big\}\big]$ 
with $|U|\geq 2$.
If $\tau(u)=0$ for some $u\in U$, 
then a simple inductive argument over ${\rm dist}_G(u,v)$ implies 
$v\in X_t$ for every $v\in U$ and $t\geq {\rm dist}_G(u,v)$,
which implies $U\subseteq X_t$ for every $t\geq {\rm diam}(G[U])$.
Hence, we may assume that $\tau(u)=1$ for every vertex $u\in U$.

Next, let $G[U]$ be non-bipartite.
By condition (1)(a), there is some vertex $u\in X\cap N_G[U]$.
Let $v_0v_1\ldots v_{2k}$ be an odd cycle in $G[U]$, and let $u_0u_1\ldots u_\ell$ 
be a path in $G[N_G[U]]$ such that $u=u_0$, $u_\ell=v_0$, and $u_1,\ldots,u_\ell\in U$.
It follows that $u_i\in X_i$ for every $i\in \{ 0,\ldots,\ell\}$,
in particular, we have $v_0\in X_{\ell}$.
Now, it follows that $v_j,v_{2k+1-j}\in X_{\ell+j}$ for every $j\in \{ 1,\ldots,k\}$,
in particular, we have $v_k,v_{k+1}\in X_{\ell+k}$.
This implies that $v_k,v_{k+1}\in X_t$ for every $t\geq \ell+k$,
and, similarly as above, it follows that $U\subseteq X_t$ for every $t\geq \ell+k+{\rm diam}(G[U])$.

Finally, let $G[U]$ be bipartite, and let $A$, $B$, $A'$, $B'$, and $C$ be as in (1)(b). 
Since $X$ contains a vertex from $A\cup A'\cup C$ 
as well as a vertex from $B\cup B'\cup C$,
the set $X_1$ contains a vertex $a$ from $A$
and a vertex $b$ from $B$.
Let $u_0u_1\ldots u_{2k+1}$ be a path in $G[U]$
between $a=u_0$ and $b=u_{2k+1}$.
It follows that $u_i,u_{2k+1-i}\in X_{1+i}$ for every $i\in \{ 0,1,\ldots,k\}$,
in particular, $u_k,u_{k+1}\in X_{1+k}$.
This implies that $u_k,u_{k+1}\in X_t$ for every $t\geq 1+k$,
and, similarly as above, it follows that $U\subseteq X_t$ 
for every $t\geq 1+k+{\rm diam}(G[U])$,
which completes the proof.
\end{proof}

Our next result concerns the hardness of instances $(G,\tau)$
satisfying (\ref{erest}).

\begin{theorem}\label{theorem2}
For every fixed positive integer $d$,
it is NP-complete to decide,
for a given triple $(G,\tau,k)$, where 
\begin{itemize}
\item $G$ is a planar graph with vertices of degree $2$ and $3$,
in which every two vertices of degree $3$ have distance at least $d$,
\item $\tau$ is a threshold function for $G$ satisfying (\ref{erest}),
and 
\item $k$ is a positive integer,
\end{itemize}
whether $(G,\tau)$
has a non-monotone target set of order at most $k$.
\end{theorem}
\begin{proof}
Theorem~\ref{theorem1} immediately implies that 
the considered decision problem is in NP.
In order to prove NP-completeness,
let ${\cal C}$ be an instance of {\sc Satisfiability}
consisting of the clauses $C_1,\ldots,C_m$
over the boolean variables $x_1,\ldots,x_n$, where 
\begin{itemize}
\item every clause contains two or three literals, 
\item for every boolean variable $x_i$, 
no clause contains both literals $x_i$ and $\bar{x}_i$,
exactly two clauses contain the literal $x_i$, and 
exactly one clause contains the literal $\bar{x}_i$, and
\item the bipartite graph with partite sets $\{ C_1,\ldots,C_m\}$ and $\{ x_1,\ldots,x_n\}$
in which $C_j$ is adjacent to $x_i$ if the clause $C_j$ contains $x_i$ or $\bar{x}_i$,
is planar.
\end{itemize}
It is well known \cite{da} that {\sc Satisfiability} 
remains NP-complete for such instances.

We now describe a polynomial time construction 
of $(G,\tau,k)$ as in the statement 
such that ${\cal C}$ is satisfiable
if and only if 
there is a non-monotone target set for $(G,\tau)$ 
of order at most $k$:
\begin{itemize}
\item For every $i\in [n]$, 
create a graph $G_i$ as shown in Figure~\ref{fig1}.
\item For every $j\in [m]$, 
create a triangle $C^j$.
\item For every $i\in [n]$ and $j\in [m]$,
if the positive literal $x_i$ appears in $C_j$, 
then add an edge between the vertex $x_i$ of $G_i$ and some vertex of $C^j$, and
if the negative literal $\bar{x}_i$ appears in $C_j$, 
then add an edge between the vertex $\bar{x}_i$ of $G_i$ and some vertex of $C^j$.
Ensure that the degrees of the vertices on $C^j$ remain at most $3$
by selecting different endpoints on $C^j$ for the at most $3$ edges towards $C^j$. 
\item Subdivide every edge $e$ of the graph constructed so far exactly 
$2\left\lfloor\frac{d}{2}\right\rfloor$ times, that is, 
replace each edge $e$ by a path $P_e$ of odd length at least $d$.
\end{itemize}
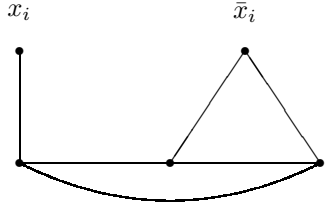
\begin{figure}[H]
\begin{center}
\unitlength 1mm 
\linethickness{0.4pt}
\ifx\plotpoint\undefined\newsavebox{\plotpoint}\fi 
\begin{picture}(41,25)(0,0)
\put(40,5){\circle*{1}}
\put(20,5){\circle*{1}}
\put(0,5){\circle*{1}}
\put(30,20){\circle*{1}}
\put(0,20){\circle*{1}}
\put(0,20){\line(0,-1){15}}
\put(0,5){\line(1,0){40}}
\put(40,5){\line(-2,3){10}}
\put(30,20){\line(-2,-3){10}}
\qbezier(0,5)(20,-5)(40,5)
\put(0,25){\makebox(0,0)[]{$x_i$}}
\put(30,25){\makebox(0,0)[]{$\bar{x}_i$}}
\end{picture}
\end{center}
\caption{The variable gadget $G_i$.} \label{fig1}
\end{figure}
This completes the description of $G$.
Note that $G$ has order $5n+3m+2\left\lfloor\frac{d}{2}\right\rfloor(9n+3m)$ 
and is as required in the statement.
It remains to specify $\tau$ and $k$:
\begin{itemize}
\item Let $\tau(x_i) = d_G(x_i)$ and $\tau(\bar{x}_i)=d_G(\bar{x}_i)$ for every $i\in [n]$, and let $\tau(v)=1$ for all remaining vertices.
\item Let $k=n$.
\end{itemize}
Note that 
$G$ has no edge $uv$ with $\tau(u)=d_G(u)$ and $\tau(v)=d_G(v)$,
and that each component of 
$G\big[\big\{ u\in V(G):\tau(u)=1\big\}\big]$ 
is of order at least $2$ and not bipartite.
Therefore, 
non-monotone target sets for $(G,\tau)$ 
are characterized by condition (1)(a)
from Theorem~\ref{theorem1}.

If ${\cal C}$ has a satisfying truth assignment $t$,
then 
$$X=\{ x_i:i\in [n]\mbox{ and $x_i$ is true under $t$}\}
\cup \{ \bar{x}_i:i\in [n]\mbox{ and $x_i$ is false under $t$}\}$$
is a non-monotone target set of order $k=n$ for $(G,\tau)$ 
by Theorem~\ref{theorem1}.

Conversely, if $X$ is a non-monotone target set 
of order at most $k=n$ for $(G,\tau)$,
then, by Theorem~\ref{theorem1} and since $k=n$,
the set $X$ contains exactly one vertex from $\bigcup\limits_{e\in E(G_i)}V(P_e)$ 
for every $i\in [n]$.
Hence, the intersections $X\cap \{ x_i,\bar{x}_i\}$
define a partial truth assignment.
Let the truth assignment $t$ extend this partial truth assignment.
Again by Theorem~\ref{theorem1},
the set $X$ contains at least one vertex from $N_G[V(C^j)]$,
which implies that $t$ is satisfying.
\end{proof}

Our next result is based on the commonly used notion of a {\it nice tree decomposition} \cite{kl}.
Let~$G$ be a graph.
A {\it tree decomposition} of $G$ is a pair $\left(T,(X_t)_{t\in V(T)}\right)$, where 
\begin{itemize}
\item $T$ is a tree, 
\item $X_t\subseteq V(G)$ for every node $t$ of $T$,
\item $\big\{ t\in V(T): u\in X_t\big\}$ induces a non-empty subtree of $T$ for every vertex $u$ of $G$, and,
\item for every edge $uv$ of $G$, there is some node $t$ of $T$ with $u,v\in X_t$.
\end{itemize}
The {\it width} of the tree decomposition is $\max\big\{|X_t|:t\in V(T)\big\}-1$.
The tree decomposition is {\it nice} if $T$ is a rooted binary tree, 
and every node $t$ of $T$ is of one of the following types:
\begin{itemize}
\item $t$ is a leaf of $T$, and $X_t=\emptyset$ ({\it leaf node}).
\item $t$ has two children $t'$ and $t''$, and $X_t=X_{t'}=X_{t''}$ ({\it join node}).
\item $t$ has a unique child $t'$, and 
\\ either $|X_t\setminus X_{t'}|=1$ and $|X_{t'}\setminus X_t|=0$ ({\it introduce node}),
\\ or $|X_{t'}\setminus X_t|=1$ and $|X_t\setminus X_{t'}|=0$ ({\it forget node}).
\end{itemize}
The proof of our next result 
is based on the reduction described in Lemma~\ref{lemma1}
and dynamic programming along a nice tree decomposition.

\begin{theorem}\label{theorem3}
Given 
\begin{itemize}
\item a pair $(G,\tau)$, where 
$G$ is a graph of order $n(G)$,
and $\tau$ is a threshold function for $G$ 
satisfying (\ref{erest}),
and 
\item a nice tree decomposition $\left(T,(X_t)_{t\in V(T)}\right)$ of $G$ of width $w$,
where $T$ has order $n(T)$,
\end{itemize}
the minimum order of a non-monotone target set for $(G,\tau)$
can be determined in time 
$$2^{5w}\cdot n(T) \cdot {\rm poly}(n(G)).$$
\end{theorem}
\begin{proof}
Let $(G,\tau)$ and $\left(T,(X_t)_{t\in V(T)}\right)$ be as in the statement.
Applying the polynomial time reduction described in Lemma~\ref{lemma1}, 
we may assume that $G$ has no edge $uv$ with $\tau(u)=d_G(u)$ and $\tau(v)=d_G(v)$.
Since this reduction only involves the removal of vertices from $G$,
the given initially nice tree decomposition
can be modified in time $n(T) \cdot {\rm poly}(n(G))$
in such a way that it stays nice.
Possibly adding $O(n(G))$ further forget nodes to $T$, 
we may assume that $X_{t_0}=\emptyset$, where $t_0$ is the root of $T$.
For every node $t$ of $T$, 
let $Z_t$ denote the set of nodes of $T$ that contains $t$ 
as well as all its descendants,
and, let $G_t$ be the subgraph of $G$ induced by $\bigcup\limits_{s\in Z_t} X_s$.

Let $G'$ arise from the graph
$$G\big[\big\{ u\in V(G):\tau(u)\in \{ 0,1\}\big\}\big]$$ 
by removing all components that have order $1$ or contain a vertex $u$ with $\tau(u)=0$.

Let 
\begin{itemize}
\item $U_1,\ldots,U_p$ be the vertex sets of the non-bipartite components of $G'$, and let 
\item $U'_1,\ldots,U'_q$ be the vertex sets of the bipartite components of $G'$.
\end{itemize}
For $U=U'_i$ for some $i\in [q]$, and $A$, $B$, $A'$, $B'$, and $C$ as in Theorem~\ref{theorem1}~(1)(b),
let 
\begin{itemize}
\item $\bar{A}(U)=A\cup A'\cup C$ and $\bar{B}(U)=B\cup B'\cup C.$
\end{itemize}
Let 
\begin{itemize}
\item $u_1,\ldots,u_r$ be the vertices $u$ of $G$ with $\tau(u)=1$ such that $\tau(v)=d_G(v)$ for every $v\in N_G(u)$.
\end{itemize}
Clearly, all of $U_1,\ldots,U_p,U'_1,\bar{A}(U_1'),\bar{B}(U_1'),\ldots,U'_q,\bar{A}(U_q'),\bar{B}(U_q'),u_1,\ldots,u_r$
can be determined in ${\rm poly}(n(G))$ time.

For every node $t$ of $T$, let 
\begin{eqnarray*}
P_t &=& \{ i\in [p]:U_i\cap X_t\not=\emptyset\},\\
Q_t &=&  \{ i\in [q]:U'_i\cap X_t\not=\emptyset\},\\
R_t &=&  \{ i\in [r]:u_i\in X_t\},\mbox{ and }\\
D_t &=& \{ u\in X_t:\tau(u)=d_G(u)\}.
\end{eqnarray*}
Let $t$ be a node of $T$.
Let $V^-=V(G_t)\setminus X_t$.
Note that $N_G(u)=N_{G_t}(u)$ for every vertex $u\in V^-$.

A {\it pattern} for $t$ is a $5$-tuple
$\left(S,B_{(1)(a)},B_{(1)(b)},B_{(2)},B_{d}\right)$, where
\begin{enumerate}[(i)]
\item $S\subseteq X_t$,
\item $B_{(1)(a)}=\big(b(U_i)\big)_{i\in P_t}\in \{ 0,1\}^{|P_t|}$,
\item $B_{(1)(b)}=\big(\big(b_{\bar{A}}(U'_i),b_{\bar{B}}(U'_i)\big)\big)_{i\in Q_t}\in \big(\{ 0,1\}^2\big)^{|Q_t|}$,
\item $B_{(2)}=\big(\big(b_{v_1}(u_i),b_{v_2}(u_i)\big)\big)_{i\in R_t}\in \big(\{ 0,1\}^2\big)^{|R_t|}$, and
\item $B_{d}=\big(b(u)\big)_{u\in D_t}\in \{ 0,1\}^{|D_t|}$.
\end{enumerate}
Intuitively, the set $S$ fixes the intersection 
of a potential non-monotone target set with $X_t$,
and the suitably formatted $0/1$-vector
$\left(B_{(1)(a)},B_{(1)(b)},B_{(2)},B_{d}\right)$
stores bits of information related to the conditions 
in Theorem~\ref{theorem1}.
For a pattern $\left(S,B_{(1)(a)},B_{(1)(b)},B_{(2)},B_{d}\right)$ for $t$,
let 
$$x\left(t,S,B_{(1)(a)},B_{(1)(b)},B_{(2)},B_{d}\right)$$
be the minimum of $|X\setminus X_t|$
over all subsets $X$ of $V(G_t)$ such that 
\begin{enumerate}[(C1)]
\item $X\cap X_t=S$.
\item $X\cap N_G[U_i]\not=\emptyset$ for every $i\in [p]$ with $U_i\subseteq V^-$.
\item $X\cap \bar{A}(U)\not=\emptyset$ and $X\cap \bar{B}(U)\not=\emptyset$ for every $i\in [q]$ with $U'_i\subseteq V^-$.
\item For every $i\in [r]$ with $u_i\in V^-$, there are vertices $v_1,v_2\in N_{G_t}(u_i)$
such that $\{ v_1\}\cup N_{G_t}(v_2)\subseteq X$, 
and, if $v_2\in X_t$, then $b(v_2)=1$.
\item $X\cap N_G[U_i]\cap V^-\not=\emptyset$ for every $i\in P_t$ with $b(U_i)=1$. 

(Note that, if $b(U_i)=0$, then $X\cap N_G[U_i]\cap V^-$ is not required to be empty.
In other words, $b(U_i)=1$ is not equivalent to $X\cap N_G[U_i]\cap V^-\not=\emptyset$,
but ``$b(U_i)=1$'' just imposes one more condition than ``$b(U_i)=0$''.)
\item $X\cap \bar{A}(U_i')\cap V^-\not=\emptyset$ for every $i\in Q_t$ with $b_{\bar{A}}(U'_i)=1$.
\item $X\cap \bar{B}(U_i')\cap V^-\not=\emptyset$ for every $i\in Q_t$ with $b_{\bar{B}}(U'_i)=1$.
\item $X\cap V^-$ contains a neighbor of $u_i$ for every $i\in R_t$ with $b_{v_1}(u_i)=1$.
\item $V^-$ contains a neighbor $v_2$ of $u_i$ with $N_{G_t}(v_2)\subseteq X$ for every $i\in R_t$ with $b_{v_2}(u_i)=1$.
\item $N_{G_t}(u)\subseteq X$ for every $u\in D_t$ with $b(u)=1$.
\end{enumerate}
If there is no set $X$ satisfying these conditions, 
then let $x\left(t,S,B_{(1)(a)},B_{(1)(b)},B_{(2)},B_{d}\right)=\infty$.

If 
$${\cal P}=\left(S,B_{(1)(a)},B_{(1)(b)},B_{(2)},B_{d}\right)
\,\,\,\,\,\,\,\,\mbox{ and }\,\,\,\,\,\,\,\,
{\cal P}'=\left(S,B'_{(1)(a)},B'_{(1)(b)},B'_{(2)},B'_{d}\right)$$ 
are two patterns for $t$
such that ${\cal P}\geq {\cal P}'$
pointwise,
and some set $X$ satisfies (C1) to (C10) for the first pattern ${\cal P}$,
then $X$ also satisfies (C1) to (C10) for the second pattern ${\cal P}'$.
This immediately implies 
$$x\left(S,B_{(1)(a)},B_{(1)(b)},B_{(2)},B_{d}\right)\geq 
x\left(S,B'_{(1)(a)},B'_{(1)(b)},B'_{(2)},B'_{d}\right).$$
Since $|P_t|+|Q_t|+|R_t|+|D_t|\leq |X_t|\leq w+1$,
there are at most 
$2^{w+1}4^{w+1}=8^{w+1}$ 
patterns for $t$.
Furthermore, if $t_0$ is the root of $T$, then, 
since $X_t=\emptyset$ and $G_t=G$,
we have 
$P_{t_0}=Q_{t_0}=R_{t_0}=D_{t_0}=\emptyset$,
and 
$x\left(t_0,\emptyset,\emptyset,\emptyset,\emptyset,\emptyset\right)$
is the minimum order of a non-monotonous target set for $(G,\tau)$.
Similarly, if $t$ is a leaf of $T$,
then $\left(\emptyset,\emptyset,\emptyset,\emptyset,\emptyset\right)$
is the only choice for 
$\left(S,B_{(1)(a)},B_{(1)(b)},B_{(2)},B_{d}\right)$,
and 
$x\left(t,\emptyset,\emptyset,\emptyset,\emptyset,\emptyset\right)=0$.
In order to complete the proof, 
it suffices to explain how to determine 
the values  
$x\left(t,S,B_{(1)(a)},B_{(1)(b)},B_{(2)},B_{d}\right)$
recursively in an efficient way
for every node $t$ of $T$ that is not a leaf.
How to obtain the stated running time is explained at the end of the proof.

\begin{claim}\label{claimjoin}
Let $t$ be a join node with the two children $t'$ and $t''$.
If $\left(S,B_{(1)(a)},B_{(1)(b)},B_{(2)},B_{d}\right)$ 
is a pattern for $t$,
then $x\left(t,S,B_{(1)(a)},B_{(1)(b)},B_{(2)},B_{d}\right)$
is the minimum value of 
\begin{eqnarray}\label{eclaimjoin}
x\left(t',S,B'_{(1)(a)},B'_{(1)(b)},B'_{(2)},B_{d}\right)
+x\left(t'',S,B''_{(1)(a)},B''_{(1)(b)},B''_{(2)},B_{d}\right),
\end{eqnarray}
where 
$\left(S,B'_{(1)(a)},B'_{(1)(b)},B'_{(2)},B_d\right)$ is a pattern for $t'$ with
$$\left(B'_{(1)(a)},B'_{(1)(b)},B'_{(2)}\right)
=\left(\big(b'(U_i)\big)_{i\in P_t},\big(\big(b'_{\bar{A}}(U'_i),b'_{\bar{B}}(U'_i)\big)\big)_{i\in Q_t},\big(\big(b'_{v_1}(u_i),b'_{v_2}(u_i)\big)\big)_{i\in R_t}\right),$$
and 
$\left(S,B''_{(1)(a)},B''_{(1)(b)},B''_{(2)},B_d\right)$ is a pattern for $t''$ with
$$\left(B''_{(1)(a)},B''_{(1)(b)},B''_{(2)}\right)
=\left(\big(b''(U_i)\big)_{i\in P_t},\big(\big(b''_{\bar{A}}(U'_i),b''_{\bar{B}}(U'_i)\big)\big)_{i\in Q_t},\big(\big(b''_{v_1}(u_i),b''_{v_2}(u_i)\big)\big)_{i\in R_t}\right)$$
such that 
\begin{itemize}
\item $b(U_i)\leq b'(U_i)+b''(U_i)$ for every $i\in P_t$.
\item $b_{\bar{A}}(U'_i)\leq b'_{\bar{A}}(U'_i)+b''_{\bar{A}}(U'_i)$ for every $i\in Q_t$.
\item $b_{\bar{B}}(U'_i)\leq b'_{\bar{B}}(U'_i)+b''_{\bar{B}}(U'_i)$ for every $i\in Q_t$.
\item $b_{v_1}(u_i)\leq b'_{v_1}(u_i)+b''_{v_1}(u_i)$ for every $i\in R_t$.
\item $b_{v_2}(u_i)\leq b'_{v_2}(u_i)+b''_{v_2}(u_i)$ for every $i\in R_t$.
\end{itemize}
\end{claim}
\begin{proof}[Proof of Claim \ref{claimjoin}]
Note that $G_t=G_{t'}\cup G_{t''}$, 
$V(G_{t'})\cap V(G_{t''})=X_t$,
and 
$V(G_t)\setminus X_t$
is the disjoint union of 
$V(G_{t'})\setminus X_{t'}$
and
$V(G_{t''})\setminus X_{t''}$.

If $X\subseteq V(G_t)$ 
satisfies (C1) to (C10)
for $\left(t,S,B_{(1)(a)},B_{(1)(b)},B_{(2)},B_{d}\right)$,
then
there are choices of 
$\left(t',S,B'_{(1)(a)},B'_{(1)(b)},B'_{(2)},B_{d}\right)$
and 
$\left(t'',S,B''_{(1)(a)},B''_{(1)(b)},B''_{(2)},B_{d}\right)$
as in the statement such that
\begin{itemize}
\item  
$X\cap V(G_{t'})$ 
satisfies (C1) to (C10)
for $\left(t',S,B'_{(1)(a)},B'_{(1)(b)},B'_{(2)},B_{d}\right)$
and 
\item $X\cap V(G_{t''})$ 
satisfies (C1) to (C10)
for $\left(t'',S,B''_{(1)(a)},B''_{(1)(b)},B''_{(2)},B_{d}\right)$.
\end{itemize}
If, for instance, $b(U_i)=1$ for some $i\in P_t$,
then, by (C5), 
the set $X\cap (V(G_t)\setminus X_t)$ contains some vertex from $N_G[U_i]$,
which either belongs to $X\cap V(G_{t'})$,
in which case $b'(U_i)$ can be set to $1$,
or to $X\cap V(G_{t''})$,
in which case $b''(U_i)$ can be set to $1$.

Conversely,
for all choices of 
$\left(S,B'_{(1)(a)},B'_{(1)(b)},B'_{(2)},B_{d}\right)$
and 
$\left(S,B''_{(1)(a)},B''_{(1)(b)},B''_{(2)},B_{d}\right)$
as in the statement, if
\begin{itemize}
\item  
$X\cap V(G_{t'})$ 
satisfies (C1) to (C10)
for $\left(t',S,B'_{(1)(a)},B'_{(1)(b)},B'_{(2)},B_{d}\right)$
and 
\item  $X\cap V(G_{t''})$ 
satisfies (C1) to (C10)
for $\left(t'',S,B''_{(1)(a)},B''_{(1)(b)},B''_{(2)},B_{d}\right)$,
\end{itemize}
then
$X'\cup X''$
satisfies (C1) to (C10)
for $\left(t,S,B_{(1)(a)},B_{(1)(b)},B_{(2)},B_{d}\right)$.

This completes the proof of the claim.
\end{proof}

\begin{claim}\label{claiminsert}
If $t$ is an insert node with child $t'$, 
$X_t\setminus X_{t'}=\{ u\}$, and 
$\left(S,B_{(1)(a)},B_{(1)(b)},B_{(2)},B_{d}\right)$ 
is a pattern for $t$ such that $x\left(t,S,B_{(1)(a)},B_{(1)(b)},B_{(2)},B_{d}\right)<\infty$,
then the following statements hold.
\begin{enumerate}[(1)]
\item If $u\not\in S$, then $b(v)=0$ for every $v\in D_t\cap N_G(u)$.
\item If $i\in P_t\setminus P_{t'}$, then $b(U_i)=0$.
\item If $i\in Q_t\setminus Q_{t'}$, then $b_{\bar{A}}(U'_i)=b_{\bar{B}}(U'_i)=0$.
\item If $i\in R_t\setminus R_{t'}$, then $b_{v_1}(u)=b_{v_2}(u)=0$.
\item If $u\in D_t$ and $b(u)=1$, then $N_{G_t}(u)\subseteq S$.
\item $x\left(t,S,B_{(1)(a)},B_{(1)(b)},B_{(2)},B_{d}\right)
=
x\left(t',S',B'_{(1)(a)},B'_{(1)(b)},B'_{(2)},B'_{d}\right),$
where
$S'=S\setminus \{ u\}$,
$B'_{(1)(a)}=B_{(1)(a)}\mid_{P_{t'}}$,
$B'_{(1)(b)}=B_{(1)(b)}\mid_{Q_{t'}}$,
$B'_{(2)}=B_{(2)}\mid_{R_{t'}}$, and
$B'_{d}=B_{d}\setminus \{ u\}$.
\end{enumerate}
\end{claim}
\begin{proof}[Proof of Claim \ref{claiminsert}]
Note that $G_{t'}=G_t-u$, 
$N_{G_t}(u)\subseteq X_t$,
and $V(G_t)\setminus X_t=V(G_{t'})\setminus X_{t'}$.
Condition (C10) clearly implies (1).
If $i\in P_t\setminus P_{t'}$, then $U_i\cap X_t=\{ u\}$,
and $V(G_t)\setminus X_t$ contains no vertex from $N_G[U_i]$,
which implies (2).
Similar arguments imply (3) and (4).
If $u\in D_t$ and $b(u)=1$, then 
$N_{G_t}(u)\subseteq X_t$ and (C10) imply (5).
Now, the stated equality (6) follows from 
$V(G_t)\setminus X_t=V(G_{t'})\setminus X_{t'}$,
which completes the proof of the claim.
\end{proof}

For the following Claims~\ref{claimforget1} to \ref{claimforget4}, 
let $t$ be a forget node of $t$ with child $t'$,
let $X_{t'}\setminus X_t=\{ u\}$, and
let $\left(S,B_{(1)(a)},B_{(1)(b)},B_{(2)},B_{d}\right)$ 
be a pattern for $t$.
By definition,
$P_t\subseteq P_{t'}$,
$Q_t\subseteq Q_{t'}$,
$R_t\subseteq R_{t'}$, and
$D_t\subseteq D_{t'}$.
We consider various patterns 
\begin{eqnarray*}
&&\left(S',B'_{(1)(a)},B'_{(1)(b)},B'_{(2)},B'_d\right)\\
&=&\left(S',\big(b'(U_i)\big)_{i\in P_{t'}},\big(\big(b'_{\bar{A}}(U'_i),b'_{\bar{B}}(U'_i)\big)\big)_{i\in Q_{t'}},\big(\big(b'_{v_1}(u_i),b'_{v_2}(u_i)\big)\big)_{i\in R_{t'}},
\big(b'(u)\big)_{u\in D_{t'}}\right)
\end{eqnarray*}
for $t'$.

\begin{claim}\label{claimforget1}
If $u\in U_i$ for some $i\in P_t$, then
$x\left(t,S,B_{(1)(a)},B_{(1)(b)},B_{(2)},B_{d}\right)$
is the minimum of the two values
\begin{itemize}
\item $x\left(t',S,B_{(1)(a)},B_{(1)(b)},B_{(2)},B_{d}\right)$ and
\item $x\left(t',S\cup \{ u\},B'_{(1)(a)},B_{(1)(b)},B_{(2)},B_{d}\right)$,
where $b'(U_i)=0$ and the remaining entries of $B'_{(1)(a)}$ 
are as in $B_{(1)(a)}$.
\end{itemize}
\end{claim}
\begin{proof}[Proof of Claim \ref{claimforget1}]
Note that $i\in P_{t'}$, 
that is, the set $X_{t'}$ contains a vertex from $U_i$ 
that is different from $u$.

If $X \subseteq V(G_t)$ 
satisfies (C1) to (C10) 
for $\left(t,S,B_{(1)(a)},B_{(1)(b)},B_{(2)},B_{d}\right)$,
then 
\begin{itemize}
\item either $u\not\in X$, and $X$ satisfies (C1) to (C10) 
for $\left(t',S,B_{(1)(a)},B_{(1)(b)},B_{(2)},B_{d}\right)$, 
\item or $u\in X$, and $X$ satisfies (C1) to (C10) 
for $\left(t',S\cup \{ u\},B'_{(1)(a)},B_{(1)(b)},B_{(2)},B_{d}\right)$.
\end{itemize}
Conversely,
if $X'\subseteq V(G_{t'})$
satisfies (C1) to (C10) 
for $\left(t',S,B_{(1)(a)},B_{(1)(b)},B_{(2)},B_{d}\right)$,
then $u\not\in X'$, and 
$X'$
satisfies (C1) to (C10) 
for $\left(t,S,B_{(1)(a)},B_{(1)(b)},B_{(2)},B_{d}\right)$.
Furthermore, 
if $X''\subseteq V(G_{t'})$
satisfies (C1) to (C10) 
for $\left(t',S\cup \{ u\},B'_{(1)(a)},B_{(1)(b)},B_{(2)},B_{d}\right)$,
then $u\in X''$, and 
$X''$
satisfies (C1) to (C10) 
for $\left(t,S,B_{(1)(a)},B_{(1)(b)},B_{(2)},B_{d}\right)$,
regardless of the value of $b(U_i)$.
These observations imply the claim.
\end{proof}

\begin{claim}\label{claimforget2}
If $u\in U_i$ for some $i\in P_{t'}\setminus P_t$, then
$x\left(t,S,B_{(1)(a)},B_{(1)(b)},B_{(2)},B_{d}\right)$
is the minimum of the two values
\begin{itemize}
\item $x\left(t',S,B'_{(1)(a)},B_{(1)(b)},B_{(2)},B_{d}\right)$, 
where 
$$b'(U_i)=
\begin{cases}
1, & \mbox{ if $S$ contains no vertex from $N_G[U_i]$, and}\\
0, & \mbox{ otherwise,}
\end{cases}
$$
and the remaining entries of $B'_{(1)(a)}$ 
are as in $B_{(1)(a)}$, and 
\item $x\left(t',S\cup \{ u\},B''_{(1)(a)},B_{(1)(b)},B_{(2)},B_{d}\right)$, 
where 
$b''(U_i)=0$ 
and the remaining entries of $B''_{(1)(a)}$ 
are as in $B_{(1)(a)}$.
\end{itemize}
\end{claim}
\begin{proof}[Proof of Claim \ref{claimforget2}]
Since $i\in P_{t'}\setminus P_t$,
the vertex $u$ is the only vertex from $U_i$ in $X_{t'}$,
which implies $N_G[U_i]\subseteq V(G_t)$.

If $X\subseteq V(G_t)$ 
satisfies (C1) to (C10) 
for $\left(t,S,B_{(1)(a)},B_{(1)(b)},B_{(2)},B_{d}\right)$,
then 
\begin{itemize}
\item either $u\not\in X$, and $S$ satisfies (C1) to (C10) 
for $\left(t',S,B'_{(1)(a)},B_{(1)(b)},B_{(2)},B_{d}\right)$, 
\item or $u\in X$, and $S\cup \{ u\}$ satisfies (C1) to (C10) 
for $\left(t',S\cup \{ u\},B''_{(1)(a)},B_{(1)(b)},B_{(2)},B_{d}\right)$.
\end{itemize}
Conversely,
if $X'\subseteq V(G_{t'})$
satisfies (C1) to (C10) 
for $\left(t',S,B'_{(1)(a)},B_{(1)(b)},B_{(2)},B_{d}\right)$,
then $u\not\in X'$, and 
$X'$
satisfies (C1) to (C10) 
for $\left(t,S,B_{(1)(a)},B_{(1)(b)},B_{(2)},B_{d}\right)$,
and,
if $X''\subseteq V(G_{t'})$
satisfies (C1) to (C10) 
for $\left(t',S\cup \{ u\},B''_{(1)(a)},B_{(1)(b)},B_{(2)},B_{d}\right)$,
then $u\in X''$, and 
$X''$
satisfies (C1) to (C10) 
for $\left(t,S,B_{(1)(a)},B_{(1)(b)},B_{(2)},B_{d}\right)$,
regardless of the value of $b(U_i)$.
These observations imply the claim.
\end{proof}

If 
$u\in U_i$ for some $i\in Q_t$
or
$u\in U_i$ for some $i\in Q_{t'}\setminus Q_t$,
then there are statements that are completely analogous 
to Claims~\ref{claimforget1} and \ref{claimforget2}, and thus, we omit the details.

\begin{claim}\label{claimforget3}
If $u=u_i$ for some $i\in R_{t'}$, then
$x\left(t,S,B_{(1)(a)},B_{(1)(b)},B_{(2)},B_{d}\right)$
equals the value of 
$x\left(t',S\cup \{ u\},B_{(1)(a)},B_{(1)(b)},B'_{(2)},B_{d}\right)$,
where
$$b'_{v_1}(u_i)=
\begin{cases}
1, & \mbox{ if $S$ contains no neighbor of $u_i$, and}\\
0, & \mbox{ otherwise,}
\end{cases}
$$
and
$$b'_{v_2}(u_i)=
\begin{cases}
1, & \mbox{ if $X_t$ contains no neighbor $v_2$ of $u_i$ with $b(v_2)=1$, and}\\
0, & \mbox{ otherwise}.
\end{cases}
$$
\end{claim}
\begin{proof}[Proof of Claim \ref{claimforget2}]
Note that $i\not\in R_t$.
The stated equality follows immediately from (C4) applied to $t$
as well as (C8) and (C9) applied to $t'$.
\end{proof}

\begin{claim}\label{claimforget4}
If $u\in D_{t'}$, then
$x\left(t,S,B_{(1)(a)},B_{(1)(b)},B_{(2)},B_{d}\right)$
is the minimum of the four values
\begin{itemize}
\item $x\left(t',S,B_{(1)(a)},B_{(1)(b)},B_{(2)},B'_{d}\right)$,
where
\begin{itemize}
\item $b'(u)=0$ and the remaining entries of $B'_d$ 
are as in $B_d$,
\end{itemize}
\item $x\left(t',S,B_{(1)(a)},B_{(1)(b)},B'_{(2)},B'_{d}\right)$,
where 
\begin{itemize}
\item $b'(u)=1$ and the remaining entries of $B'_d$ 
are as in $B_d$,
\item $b'_{v_2}(u_i)=0$ for every $i\in R_t$ such that $u$ is a neighbor of $u_i$, and 
the remaining entries of $B'_{(2)}$ 
are as in $B_{(2)}$,
\end{itemize}
\item $x\left(t',S\cup \{ u\},B'_{(1)(a)},B'_{(1)(b)},B'_{(2)},B'_{d}\right)$,
where 
\begin{itemize}
\item $b'(u)=0$ and the remaining entries of $B'_d$ are as in $B_d$,
\item $b'(U_i)=0$ for every $i\in P_t$ with $u\in N_G[U_i]$ and 

the remaining entries of $B'_{(1)(a)}$ are as in $B_{(1)(a)}$,
\item $b'_{\bar{A}}(U'_i)=0$ for every $i\in Q_t$ with $u\in \bar{A}(U'_i)$,

$b'_{\bar{B}}(U'_i)=0$ for every $i\in Q_t$ with $u\in \bar{B}(U'_i)$,
and 

the remaining entries of $B'_{(1)(b)}$ are as in $B_{(1)(b)}$,
\item $b'_{v_1}(u_i)=0$ for every $i\in R_t$ 
such that $u$ is a neighbor of $u_i$, 
and 

the remaining entries of $B'_{(2)}$ are as in $B_{(2)}$.
\end{itemize}
\item $x\left(t',S\cup \{ u\},B'_{(1)(a)},B'_{(1)(b)},B'_{(2)},B'_{d}\right)$,
where 
\begin{itemize}
\item $b'(u)=1$ and the remaining entries of $B'_d$ are as in $B_d$,
\item $b'(U_i)=0$ for every $i\in P_t$ with $u\in N_G[U_i]$ and 

the remaining entries of $B'_{(1)(a)}$ are as in $B_{(1)(a)}$,
\item $b'_{\bar{A}}(U'_i)=0$ for every $i\in Q_t$ with $u\in \bar{A}(U'_i)$,

$b'_{\bar{B}}(U'_i)=0$ for every $i\in Q_t$ with $u\in \bar{B}(U'_i)$,
and 

the remaining entries of $B'_{(1)(b)}$ are as in $B_{(1)(b)}$,
\item $b'_{v_1}(u_i)=0$ for every $i\in R_t$ 
such that $u$ is a neighbor of $u_i$, 

$b'_{v_2}(u_i)=0$ for every $i\in R_t$ 
such that $u$ is a neighbor of $u_i$, 
and 

the remaining entries of $B'_{(2)}$ are as in $B_{(2)}$.
\end{itemize}
\end{itemize}
\end{claim}
\begin{proof}[Proof of Claim \ref{claimforget4}]
Note that 
$P_t=P_{t'}$,
$Q_t=Q_{t'}$, 
$R_t=R_{t'}$, and
$D_{t'}=D_t\cup \{ u\}$.
The four cases correspond to the four possibilites
\begin{itemize}
\item $u\not\in X$ and $b(u)=0$,
\item $u\not\in X$ and $b(u)=1$,
\item $u\in X$ and $b(u)=0$, and
\item $u\in X$ and $b(u)=1$,
\end{itemize}
and they encode the consequences 
for  
$U_i$ with $i\in P_t$,
$U'_i$ with $i\in Q_t$, and
$u_i$ with $i\in R_t$ 
for those elements affected by $u$.
Similar obvious observations as in the proof of Claim~\ref{claimforget2}
complete the proof of this claim.
\end{proof}

\begin{claim}\label{claimforget5}
If 
$u\not\in 
\bigcup\limits_{i\in P_{t'}}U_i
\cup
\bigcup\limits_{i\in Q_{t'}}U'_i
\cup
\bigcup\limits_{i\in R_{t'}}\{ u_i\}
\cup D_{t'}$, then
$x\left(t,S,B_{(1)(a)},B_{(1)(b)},B_{(2)},B_{d}\right)$
is the minimum of the two values
\begin{itemize}
\item $x\left(t',S,B_{(1)(a)},B_{(1)(b)},B_{(2)},B_{d}\right)$ and
\item $x\left(t',S\cup \{ u\},B_{(1)(a)},B_{(1)(b)},B_{(2)},B_{d}\right)$.
\end{itemize}
\end{claim}
\begin{proof}[Proof of Claim \ref{claimforget5}]
This follows immediately from the definitions.
\end{proof}

In order to complete the proof, 
it suffices to argue that, 
spending $2^{5w}\cdot {\rm poly}(n(G))$ time for each of the $n(T)$ nodes $t$ of $T$,
all values of $x\left(t,S,B_{(1)(a)},B_{(1)(b)},B_{(2)},B_{d}\right)$
can be determined.
Since the initialization of the leaves is trivial,
processing the nodes of $T$ from the leaves to the root,
we may assume, for each node $t$ currently considered,
that we dispose of all values for its one or two children.
Considering the cases corresponding to the different claims,
it is easy to see, that the join nodes considered in Claim~\ref{claimjoin}
entail most effort, and we give details only for these.
Therefore, let $t$ be a join node.
\begin{itemize}
\item We initialize all values $x(t,\ldots)$ values as $\infty$.
\item We loop through all at most $2^{5(w+1)}$ choices for 
$$\left(S,B'_{(1)(a)},B'_{(1)(b)},B'_{(2)},B''_{(1)(a)},B''_{(1)(b)},B''_{(2)},B_{d}\right)$$
using the notation of Claim~\ref{claimjoin}, and update 
$x\left(t,S,B_{(1)(a)},B_{(1)(b)},B_{(2)},B_{d}\right)$
with the minimum of its current value and
$$x\left(t',S,B'_{(1)(a)},B'_{(1)(b)},B'_{(2)},B_{d}\right)
+x\left(t'',S,B''_{(1)(a)},B''_{(1)(b)},B''_{(2)},B_{d}\right),$$
where every entry $b$ of $\left(B_{(1)(a)},B_{(1)(b)},B_{(2)}\right)$
satisfies 
$$b=\min\{ 1,b'+b''\}$$
for the two corresponding entries $b'$ and $b''$ 
of $\left(B'_{(1)(a)},B'_{(1)(b)},B'_{(2)}\right)$
and 
$\left(B''_{(1)(a)},B''_{(1)(b)},B''_{(2)}\right)$,
respectively.
\item Now, we loop through at most $2^{3(w+1)}$ choices for 
$${\cal P}=\left(S,B_{(1)(a)},B_{(1)(b)},B_{(2)},B_d\right)$$
as in Claim~\ref{claimjoin}, 
in lexicographically increasing order of $\left(B_{(1)(a)},B_{(1)(b)},B_{(2)},B_d\right)$.

For each choice of ${\cal P}=\left(S,B_{(1)(a)},B_{(1)(b)},B_{(2)},B_d\right)$,
we loop through all  
at most $2^{2(w+1)}$ choices for $\left(B'_{(1)(a)},B'_{(1)(b)},B'_{(2)}\right)$
such that 
$\left(B_{(1)(a)},B_{(1)(b)},B_{(2)}\right)\geq \left(B'_{(1)(a)},B'_{(1)(b)},B'_{(2)}\right)$
pointwise,
and update the value of $x\left(S,B'_{(1)(a)},B'_{(1)(b)},B'_{(2)},B'_d\right)$
with the minimum of its current value and 
the value of $x\left(S,B_{(1)(a)},B_{(1)(b)},B_{(2)},B_d\right)$.
\end{itemize}
By Claim~\ref{claimjoin} and the comments preceding it,
all values $x(t,\ldots)$ are correct after the completion of these loops,
which completes the proof.
\end{proof}

If only $(G,\tau)$ is given, and $G$ has order $n$ and treewidth $w$,
then one can, in time $2^{\mathcal{O}(w)}n$ \cite{bodrdrfolopi,kl},
determine a nice tree decomposition of $G$ 
of width at most $\mathcal{O}(w)$ such that 
the underlying tree has order $\mathcal{O}(wn)$.
This immediately implies our final result.

\begin{corollary}\label{corollary1}
Given 
a pair $(G,\tau)$, where 
$G$ is a graph of order $n$ and treewidth $w$,
and $\tau$ is a threshold function for $G$ 
satisfying (\ref{erest}),
the minimum order of a non-monotone target set for $(G,\tau)$
can be determined in time 
$2^{\mathcal{O}(w)}{\rm poly}(n(G))$.
\end{corollary}

\end{document}